\documentclass[a4paper]{article}

\usepackage{amsfonts}     
\usepackage{amsmath}
\usepackage{amssymb}
\usepackage{amstext}
\usepackage{amsthm}

\newcommand{\R}{\mathbb R}
\newcommand{\N}{\mathbb N}

\newcommand{\eps}{\varepsilon}

\newtheorem{teo}{Theorem}
\newtheorem{cor}{Corollary}[teo]
\newtheorem{lemma}[teo]{Lemma}
\newtheorem{prop}[teo]{Proposition}

\newtheorem{defin}{Definition}

\theoremstyle{definition}
\newtheorem{ex}{Example}

\theoremstyle{remark}

\newcommand{\h}{{\cal H}}
\DeclareMathOperator{\cat}{cat}

\begin{document}
\begin{centering}

{\Large\sc Geodesics in Conical Manifolds}

\vskip 0.5cm
\noindent{\sc Marco Ghimenti}\\
{\small Dipartimento di Matematica, Universit\`a di Pisa,\\ v. Buonarroti 2, 56100, Pisa, Italy}\\
{\small Dipartimento di Matematica Applicata, Universit\`a di Pisa, \\ via Bonanno 25b, 56100, Pisa, Italy}

\end{centering}

\begin{abstract}
The aim of this paper is to extend the definition of geodesics to
conical manifolds, defined as submanifolds of $\R^n$ with a
finite number of singularities. We look for an
approach suitable both for the local geodesic problem and for the
calculus of variation in the large. We give a definition
which links the local solutions of
the Cauchy problem (\ref{PbCau}) with variational
geodesics, i.e. critical points of the energy functional.
We prove a deformation lemma (Theorem \ref{maindeflemma}) which
leads us to extend the Lusternik-Schnirelmann theory to
conical manifolds, and to estimate the number of
geodesics (Theorem \ref{numpunticrit} and Corollary
\ref{catmaincor}). In section \ref{sezappl}, we provide some
applications in which conical manifolds arise naturally: in
particular, we focus on the brachistochrone problem for a
frictionless particle moving in $S^n$ or in $\R^n$ in the presence of
a potential $U(x)$ unbounded from below. We conclude with an
appendix in which the main results are presented in a
general framework.

\end{abstract}

\section{Introduction and basic definition}

The existence of geodesic is one of most studied
problems in the calculus of variation. In this paper we want to study
the presence of geodesics in a particular kind of manifolds,
called conical manifolds, that
appears in a natural way in some optimization problem (see section
\ref{sezbrac})

We define the following type of topological manifolds.
\begin{defin}
A conical manifold $M$ is a complete $n$-dimensional $C^0$ sub manifold
of $\R^m$ which is everywhere
smooth, except for a finite set of points $V$. A point in $V$ is
called vertex.
\end{defin}

Usually there are two ways to introduce geodesics in a smooth manifold:

\begin{description}
\item[Local (Cauchy problem):]{a geodesic is a solution of a
suitable Cauchy problem, i.e. given $p\in M$, $v\in T_pM$, we look
for a curve $\gamma:[0,\eps]\rightarrow M$ s.t.
\begin{equation}\label{PbCau}
\left\{
\begin{array}{l}
D_s\gamma'=0;\\
\gamma(0)=p;\\
\gamma'(0)=v.
\end{array}
\right.
\end{equation}}
\item[Global (Bolza problem):]{we consider the path space on $M$:
\begin{eqnarray*}
&&\Omega_{p,q}:=
\left\{
\gamma\in H^1([0,1],M)\;:\;\gamma(0)=p,\,\gamma(1)=q
\right\};\\
&&\Omega_{p}:=
\left\{
\gamma\in H^1([0,1],M)\;:\;\gamma(0)=\gamma(1)=p
\right\};
\end{eqnarray*}
a geodesic is a critical point of the energy functional defined
by\footnote{Hereafter we simply note $\Omega$
when we not need to specify the extremal points
of paths.}
\begin{eqnarray*}
&&E:\Omega\rightarrow\R\\
&&E(\gamma)=\int_0^1|\gamma'(s)|^2ds
\end{eqnarray*}
}
\end{description}

In conical manifolds the Cauchy problem (\ref{PbCau}) is not well posed,
and the solution is  neither unique, nor continuously dependent from
the initial data. The functional approach gives us an
easy result on minimal geodesics. However, this approach is not
completely useful: we can not  easily define a critical point
of energy different from minimum, because the energy is not a
$C^1$ functional.

Furthermore, the usual generalization of the derivative, the
weak slope, cannot be applied to our case, because it requires
some conditions on the manifolds $M$ which are not satisfied in
the case of conical manifolds. The weak slope was introduced by
Marco Degiovanni and Marco Marzocchi in \cite{DM94} (see also
\cite{De96,CD94,CDM93}). Moreover we refer to
\cite{DM99,MM02} for a weak slope approach to geodesic problem
and to \cite{Ghi04} for a detailed comparison with our approach.

We give the following definition of
geodesics, that appears to be the most suitable one for this kind
of problem.
\begin{defin}
\label{defgeod} A path $\gamma\in \Omega$ is a
geodesic iff
\begin{itemize}
\item{ the set $T=T_\gamma:=\{s\in(0,1):\gamma(s)\in V\}$ is a
closed set without internal part;} \item $D_s\gamma'=0\,\,\forall
s\in[0,1]\smallsetminus T$; \item $|\gamma'|^2$ is constant as a
function in $L^1$.
\end{itemize}
\end{defin}
We note that a geodesic may not be a local minimum for the length
functional, for example, we consider a Euclidean cone and a broken
geodesic passing through the vertex. However, this definition
allows us to prove the main theorem of this paper (see corollary
\ref{catmaincor})
\begin{teo}
Let $M$ be a conical manifold, $p\in M$. Then there are at least
$\cat\Omega$ geodesics.
\end{teo}
We are relating definition \ref{defgeod}, which is local, with
the topology of the path space, which is a tool of the calculus of
variation in the large; furthermore, this approach allows us to
find also non minimal geodesics.

Unfortunately, it's not easy to compute $\cat\Omega$ for conical
manifolds. Set
\begin{displaymath}
\Omega^\infty_{p,q}:=
\left\{
\gamma\in C^0([0,1],M)\;:\;\gamma(0)=p,\,\gamma(1)=q
\right\};
\end{displaymath}
we know that, for a smooth manifold, there is an homotopy
equivalence
\begin{equation}
\label{cap3homeqv}
\Omega^\infty_{p,q}\simeq\Omega_{p,q}
\end{equation}
(for a proof see, for example \cite[Th 1.2.10]{Kli78}).
In general this result is false for conical manifolds;
we show it by an example.
\begin{ex}
Let
\begin{displaymath}
M=\left\{ \left( x,x\sin \frac{1}{x}\right),\;x\in\R
\right\}\subset\R^2;
\end{displaymath}
this is an 1-dimensional conical manifold with vertex $O=(0,0)$.
Let $p,q\in M$ be two opposite points with respect to $O$: we have that, while
$\Omega^\infty_{p,q}$is connected, $\Omega_{p,q}$ is not,
so the usual homotopy
equivalence \ref{cap3homeqv} does not hold.
\end{ex}
Even if
an explicit calculation of $\cat\Omega$ in general is very difficult,
in section \ref{sezappl}, we will give a criterion
for which (\ref{cap3homeqv})
holds. Moreover
we show some applications in which conical manifolds
appears naturally.

\section{Deformation lemmas}
We want to prove that our definition of geodesic is compatible
with the energy functional, i.e. if there is no geodesic of energy
$c$, then there is no change of the topology of functional $E$ at
level $c$. To do that, we prove a deformation lemma
(Theorem \ref{maindeflemma}), that is the main result of this
section.

\begin{defin}
Given $p\in M$ we set
\begin{eqnarray*}
&&\Omega^b=\Omega_p^b:= \left\{ \gamma\in \Omega_p\;:\;E(\gamma)\leq b
\right\};\\
&&\Omega_a^b=\Omega_{a,p}^b:= \left\{ \gamma\in
\Omega_p\;:\;a\leq E(\gamma)\leq b \right\}.
\end{eqnarray*}
\end{defin}

\begin{teo}[Deformation lemma]\label{maindeflemma}
Let $M$ be
a conical manifold, $p\in M$. Suppose that there exists $c\in\R$ s.t.
$\Omega^c$ contains only a finite number of geodesics. Then if $a,b\in \R$,
and $a<b<c$ are s.t. the strip $[a,b]$ contains only regular values of $E$,
$\Omega^a$ is a deformation retract of $\Omega^b$.
\end{teo}

In order to prove this theorem, we must study the structure of
$\Omega^c$. For the moment, we consider a special case.

We suppose that $M$ has only a
vertex $v$, and we study the special closed geodesic
$\gamma_0$ for which there exists an unique $\sigma$
s.t. $\gamma_0(\sigma)=v$. We set $E(\gamma_0)=c_0$ and we
suppose that there exist $a,b\in\R$, $c_0<a<b$, s.t. $\Omega^b$
contains only the geodesics
$\gamma_0$ (so $\Omega_a^b$ contains no geodesics).

At last let us set

\begin{displaymath}
L_1=\int_0^\sigma|\gamma'_0|^2\;\;,\;\;
L_2=\int_\sigma^1|\gamma'_0|^2.
\end{displaymath}

We identify now two special subsets of
$\Omega_a^b$.
Let
\begin{equation}
\label{sigma}
\Sigma=\{\gamma\in \Omega_a^b,
\text{ s.t. } v\in \text{Im}\gamma\};
\end{equation}
for every $\gamma\in\Sigma$ it exists a set $T$ s.t. $\gamma(s)=v$
iff $s\in T$. Let

\begin{equation}
\label{sigma0}
\Sigma_0=\left\{
\begin{array}c
\gamma\in \Sigma,\text{ s.t. } D_s\gamma'(s)=0 \text{ and
}|\gamma'|^2
\text{ is constant} \\
\text{ on every connected component of }[0,1]\smallsetminus T
\end{array}
\right\}.
\end{equation}
Indeed, we will see in the proof of the next lemma that,
if $\gamma\in \Sigma_0$, then $\gamma([0,1])=\gamma_0([0,1])$, so
$\Sigma_0$ is the set of the piecewise geodesics that
are equivalent to $\gamma_0$
up to affine reparametrization.

\begin{lemma}\label{comp}
$\Sigma_0$ is compact.
\end{lemma}
\begin{proof}
If $\gamma\in\Sigma_0$, only two situations occur: either $\exists
! \tau$ s.t. $\gamma(\tau)=v$, or $\exists [\tau_1,\tau_2]$ s.t.
$\gamma(t)=v$ iff $t\in[\tau_1,\tau_2]$. In fact, if it were two
isolated consecutive points $s_1,s_2\in T$ s.t $\gamma(s_i)=v$,
then,we can obtain by reparametrization a geodesic
$\gamma_1\neq\gamma_0$ in $\Omega^b$,
that contradicts our assumptions (this proves also that
$\gamma([0,1])=\gamma_0([0,1])$).

Now take $(\gamma_n)_n\subset\Sigma_0$. For simplicity we can
suppose that there exists a subsequence such that $\forall
n\exists\,!\tau_n$ for which $\gamma_n(\tau_n)=v$ (else,
definitely, $\exists\,[\tau_n^1,\tau_n^2]$ s.t. $\gamma_n(t)=v$
iff $s \in [\tau_n^1,\tau_n^2] $, but the proof follows in the
same way).

If we consider
$||\gamma||_{H^1}=E(\gamma)$, then we have
\begin{displaymath}
a\leq ||\gamma_n||\leq b,
\end{displaymath}
hence, up to subsequence, $\exists\, \bar\gamma$ s.t
$\gamma_n\rightarrow\bar\gamma$ in weak-$H^1$ norm ad uniformly.
Also, we know that $\forall n \;\exists\, \tau_n$ s.t.
$\gamma_n(\tau_n)=v$ and
\begin{equation}
\gamma_n=\left\{
\begin{array}{ll}
\gamma_0\left(\frac{\sigma}{\tau_n}s\right),&
s\in[0,\tau_n]\\
\gamma_0\left(\frac{1-\sigma}{1-\tau_n}s+
\frac{\sigma-\tau_n}{1-\tau_n}\right),&
s\in(\tau_n,1]
\end{array}
\right.
\end{equation}

It exists $0<p<1$ such that $p\leq\tau_n\leq1-p$, in fact
\begin{eqnarray*}
b&\geq&\int |\gamma'_n|^2= \int_0^{\tau_n}|\gamma'_n|^2+
\int_{\tau_n}^1|\gamma'_n|^2=\\
&=&\left[\frac{\sigma}{\tau_n}\right]^2
\int_0^{\tau_n}|\gamma'_0|^2\left(\frac{\sigma}{\tau_n}s\right)+\\
&&+\left[\frac{1-\sigma}{1-\tau_n}\right]^2
\int_{\tau_n}^1|\gamma'_0|^2\left(\frac{1-\sigma}{1-\tau_n}s+
\frac{\sigma-\tau_n}{1-\tau_n}\right)ds=\\
&=&\frac{\sigma}{\tau_n}\int_0^{\sigma}|\gamma'_0|^2(s')ds'+
\frac{1-\sigma}{1-\tau_n}\int_0^{\sigma}|\gamma'_0|^2(s')ds'=\\
&=&\frac{L_1^2}{\sigma\tau_n}+\frac{L_2^2}{(1-\sigma)(1-\tau_n)}\;,
\end{eqnarray*}
so
\begin{equation}
b\geq\frac{L_1^2}{\sigma\tau_n}
\;\Rightarrow\;\tau_n>\frac{L_1^2}{\sigma b}\,,
\end{equation}
and
\begin{equation}
b\geq\frac{L_2^2}{(1-\sigma)(1-\tau_n)}
\;\Rightarrow\;\tau_n<1-\frac{L_2^2}{b(1-\sigma)}\,.
\end{equation}
So a subsequence exists such that $\tau_n\rightarrow \tau$,
$p\leq\tau\leq1-p$. Obviously
\begin{eqnarray*}
\frac{\sigma}{\tau_n}s&\rightarrow&\frac{\sigma}{\tau}s\,,\\
\frac{1-\sigma}{1-\tau_n}s+\frac{\sigma-\tau_n}{1-\tau_n}
&\rightarrow&
\frac{1-\sigma}{1-\tau}s+\frac{\sigma-\tau}{1-\tau}\,.
\end{eqnarray*}
So for almost all $s$ we have
\begin{equation}\label{conv}
\gamma_n\rightarrow\tilde\gamma(s)=
\left\{
\begin{array}{ll}
\gamma_0\left(\frac{\sigma}{\tau}s\right),&s\in[0,\tau]\\
\gamma_0\left(\frac{1-\sigma}{1-\tau}s+\frac{\sigma-\tau}{1-\tau}\right),
&s\in(\tau,1]
\end{array}
\right.
\;.
\end{equation}
Both $\gamma_n$ and $\tilde\gamma$ are continuous, because $\gamma_0$ is
continuous, so the convergence in (\ref{conv}) is uniform; furthermore,
$\bar\gamma=\tilde\gamma$ for
the uniqueness of limit.

We have also that
\begin{equation}
||\gamma_n||=\frac{L_1^2}{\sigma\tau_n}+\frac{L_2^2}{(1-\sigma)(1-\tau_n)}
\rightarrow\frac{L_1^2}{\sigma\tau}+\frac{L_2^2}{(1-\sigma)(1-\tau)}=
||\bar\gamma||,
\end{equation}
so $\gamma_n\stackrel{H^1}{\rightarrow}\bar\gamma$ and
$a\leq||\bar\gamma||\leq b$, hence $\bar\gamma\in\Sigma_0$,
that concludes the proof.
\end{proof}

Now we shall prove two technical lemmas which are crucial for this
paper.

\begin{lemma}[existence of retraction in $\Sigma_0$]\label{retrsigma}
There exist $R\supset \Sigma_0$, $\nu,\bar t\in\R^+$ and
$$
\eta_R:R\times[0,\bar t\,]\rightarrow \Omega
$$
a continuous function s.t.

\begin{itemize}
\item $\eta_R(\beta,0)=\beta$,
\item $E(\eta_R(\beta,t))-E(\beta)<-\nu t$,
\end{itemize}
for all $t\in [0,\bar t\,]$, $\beta\in R$.
\end{lemma}

\begin{proof}We proceed by steps.

I) At first we want to prove that, for any $\gamma\in\Sigma_0$, there are
$\bar t,d,\nu\in\R^+$, and a local retraction
$$
\h:B(\gamma,d)\times[0,\bar t\,]\rightarrow \Omega
$$
such that

\begin{itemize}
\item $\h(\beta,0)=\beta$,
\item $E(\h(\beta,t))-E(\beta)<-\nu t$,

\end{itemize}
for all $t\in [0,\bar t\,]$, $\beta\in B(\gamma,d)$.
Furthermore, we will see that $d$ is independent from $\gamma$.

By hypothesis there exists an unique $\sigma\in[0,1]$ such that
$\gamma_0(\sigma)=~v$; furthermore, because $E(\gamma_0)=c_0$, we
know also that $|\gamma'_0|^2=c_0$ almost everywhere. Let
$\gamma\in\Sigma_0$, then Im$(\gamma)$=Im$(\gamma_0)$. In analogy
with Lemma \ref{comp} we suppose, without loss of generality, that
there exists an unique $\tau\in[0,1]$ s.t. $\gamma(\tau)=v$, and
both $\gamma'|_{(0,\tau)}$, $\gamma'|_{(\tau,1)}$ are constant,
although we cannot say if they are equals. We can choose a
suitable change of parameter $\varphi$ s.t.
\begin{equation}
\gamma(\varphi(s))=\gamma_0(s).
\end{equation}
By this way we can construct a flow for $\gamma$ as follows:
\begin{equation}
\h(\gamma,t)=\gamma(\varphi_t(s))=
\left\{
\begin{array}{ll}
\gamma\left(\frac{\tau}{a(t)}s\right)&s\in[0,a(t))\\
\gamma\left(\frac{\tau-1}{a(t)-1}s+\frac{a(t)-\tau}{a(t)-1}\right)
&s\in[a(t),1]
\end{array}
\right.
\end{equation}
where $a(t)=(1-t)\tau+t\sigma$.

Notice that
\begin{displaymath}
\gamma\big(\varphi_0(s)\big)=\gamma(s),\,
\gamma\big(\varphi_1(s)\big)=\gamma_0(s)
\end{displaymath}
and
\begin{displaymath}
\gamma(\tau)=v=\gamma\big(\varphi_1(\sigma)\big)=
\gamma\big(\varphi_t(a(t))\big).
\end{displaymath}

We recall that $l\left(\gamma_0|_{(0,\sigma)}\right)=L_1$,
$l\left(\gamma_0|_{(\sigma,1)}\right)=L_2$: obviously
\begin{displaymath}
\left[\frac{L_1}{\sigma}\right]^2=
\left[\frac{L_2}{1-\sigma}\right]^2=c_0;
\end{displaymath}
furthermore
\begin{displaymath}
\left[\frac{\partial}{\partial s}
\gamma(\varphi_t(s))\right]^2_{\Big|_{(0,a(t))}}=
\left[\frac{L_1}{a(t)}\right]^2,
\end{displaymath}
\begin{displaymath}
\left[\frac{\partial}{\partial s}
\gamma(\varphi_t(s))\right]^2_{\Big|_{(a(t),1)}}=
\left[\frac{L_2}{1-a(t)}\right]^2,
\end{displaymath}
then
\begin{eqnarray*}
E(\h(\gamma,t))&=&\int_0^{a(t)}\frac{L_1^2}{a(t)^2}ds+
\int_{a(t)}^1\frac{L_2^2}{(1-a(t))^2}ds=\\
&=&\frac{L_1^2}{\sigma^2}\frac{\sigma^2}{a(t)}+
\frac{L_2^2}{(1-\sigma)^2}\frac{(1-\sigma)^2}{1-a(t)}=
c_0\left(\frac{\sigma^2}{a(t)}+\frac{(1-\sigma)^2}{1-a(t)}
\right),
\end{eqnarray*}
so
\begin{equation}
\frac{\partial}{\partial t}E(\h(\gamma,t))=
c_0\left(\frac{(1-\sigma)^2}{(1-a(t))^2}-\frac{\sigma^2}{(a(t))^2}\right).
\end{equation}

It's easy to see that, either for $\sigma<\tau$ as for $\sigma>\tau$, we have
$\frac{\partial}{\partial t}E(\h(\gamma,t))<~0$, for all $t\in[0,1)$, as
expected. More over, because there is a $p>0$ s.t $p<\tau<1-p$ (as shown in
the previous lemma), we can find $\bar t, \nu$ s.t.
$\frac{\partial}{\partial t}E(\h(\gamma,t))<2\nu$
$\forall t\in [0,\bar t\,]$.

Now we want to extend $\h$ in a neighborhood of $\gamma$: it's
useful, for finding that, to work on the whole space
$H^1(I,\R^n)$. As above we consider $\gamma\in \Omega$.

Let $B_d=B^{H^1(I,\R^n)}(\gamma,d)\cap \Omega$.
For all $\beta\in B_d$ we can say
\begin{displaymath}
\beta=\gamma+(\beta-\gamma)=\gamma+\delta,\,\, ||\delta||\leq d
\end{displaymath}
We can extend $\h$ as follows:
\begin{equation}
\h(\beta,t)=\h(\gamma+\delta,t)=\gamma(\varphi_t(s))+\delta(\varphi_t(s))
\end{equation}
Obviously Im($\beta$)=Im($\h(\beta,t)$), so $\h(\beta,t)\in \Omega$.

We want to show that there exists a $d>0$ s.t.
\begin{equation}
E(\h(\beta,t))-E(\beta)<-\nu t\,\,\forall \beta\in B_d
\end{equation}

\begin{eqnarray*}
E(\h(\beta,t))-E(\beta)&=&\\
&=&\int|{\gamma(\varphi_t(s))'}+{\delta(\varphi_t(s))'}|^2-
\int|\gamma' (s)+\delta'(s)|^2=\\
&=&\int|{\gamma(\varphi_t(s))'}|^2-|\gamma' (s)|^2+
\int|{\delta(\varphi_t(s))'}|^2-|\delta'(s)|^2+\\
&&+\int<{\gamma(\varphi_t(s))'},{\delta(\varphi_t(s))'}>
-\int<\gamma' (s),\delta'(s)>.
\end{eqnarray*}
We have already shown that
\begin{equation}
\int|{\gamma(\varphi_t(s))}'|^2-|\gamma' (s)|^2<-2\nu t.
\end{equation}

Let
$$
A=\int_0^1|{\delta(\varphi_t(s))}'|^2ds-\int_0^1|\delta(s)'|^2ds
$$
and
$$
B=\int_0^1<{\gamma(\varphi_t(s))}',{\delta(\varphi_t(s))}'>ds
-\int_0^1<\gamma' (s),\delta'(s)>ds.
$$
The term A can be estimate as follows, remembering the definition
of $\varphi_t(s)$:

\begin{eqnarray*}
A&=&\\
&=&\int_0^{a(t)}\left|{\delta\left(\frac{\tau}{a(t)}s\right)}'\right|^2
+\int_{a(t)}^1\left|{\delta\left(\frac{\tau-1}{a(t)-1}s+
\frac{a(t)-\tau}{a(t)-1}\right)}'\right|^2-
\int_0^1|\delta'(s)|^2=\\
&=&\left[\frac{\tau}{a(t)}\right]^2
\int_0^{a(t)}\left|\delta'\left(\frac{\tau}{a(t)}s\right)\right|^2+\\
&&+\left[\frac{\tau-1}{a(t)-1}\right]^2
\int_{a(t)}^1\left|\delta'\left(\frac{\tau-1}{a(t)-1}s+
\frac{a(t)-\tau}{a(t)-1}\right)\right|^2-
\int_0^1|\delta'(s)|^2=\\
&=&\frac{\tau}{a(t)}\int_0^\tau|\delta'(s)|^2+
\frac{\tau-1}{a(t)-1}\int_\tau^1|\delta'(s)|^2-\int_0^1|\delta'(s)|^2=\\
&=&\frac{\tau-a(t)}{a(t)}\int_0^\tau|\delta'(s)|^2+
\frac{\tau-a(t)}{a(t)-1}\int_\tau^1|\delta'(s)|^2
\leq\\
&\leq&\max\left[\left|\frac{\tau-a(t)}{a(t)}\right|,
\left|\frac{\tau-a(t)}{a(t)-1}\right|\right]
\int_0^1|\delta'(s)|^2\leq\\
&\leq&|\tau-a(t)|
\max\left[\frac{1}{a(t)},\frac{1}{a(t)-1}\right]\int_0^1|\delta'(s)|^2\leq\\
&\leq&K||\delta||^2_{H^1}t\leq Kd^2\cdot t
\end{eqnarray*}
in fact $|\tau-a(t)|=|\tau-\sigma|t$. Furthermore, $K$ depends
only on $\gamma_0$, because $\exists p>0$ s.t.
$\tau\in[p,1-p]$ (as shown in Lemma \ref{comp}).

In the same way we can estimate $B$:

\begin{eqnarray*}
B&=&\\
&=&\frac{\tau-a(t)}{a(t)}\int_0^\tau<\gamma'(s),\delta'(s)>+
\frac{\tau-a(t)}{a(t)-1}\int_\tau^1<\gamma'(s),\delta'(s)>\leq\\
&\leq&\left|\frac{\tau-a(t)}{a(t)}\right|
\int_0^\tau\left|<\gamma'(s),\delta'(s)>\right|+
\left|\frac{\tau-a(t)}{a(t)-1}\right|
\int_\tau^1\left|<\gamma'(s),\delta'(s)>\right|\leq\\
&\leq&t|\tau-\sigma|\max\left[\frac{1}{a(t)},\frac{1}{1-a(t)}\right]
\int_0^1<\gamma',\delta'>\leq\\
&\leq&K_1d\cdot t,
\end{eqnarray*}
where, as above, $K_1$ is a constant depending only on $\gamma_0$.

Now, putting together all the pieces we have
\begin{equation}
E(\h(\beta,t))-E(\beta)\leq-2\nu t +Kd^2t+K_1dt<-\nu t
\end{equation}
if $d<\min\left(\frac{\nu}{K_1},\sqrt{\frac{\nu}{K}}\right)$.

II) We want to prove that, for all $\eps$ it exist a $0<\tilde t<\bar t$
s.t.
\begin{equation}
\h(B(\beta,d'),t)\subset B(\beta,(1+\eps)d')
\end{equation}
if $B(\beta,d')\subset B(\gamma,d)$, $t<\tilde t$. We start proving that,
for any $\beta,\beta_1\in B(\gamma,d)$,
\begin{eqnarray*}
||\h(\beta,t)-\h(\beta_1,t)||^2_{H^1} &\leq&
\left(\frac{\tau}{a(t)}\right)^2
\int_0^{a(t)}|\beta'-\beta'_1|^2\left(\frac{\tau}{a(t)}s\right)ds+\\
&&+\left(\frac{\tau-1}{a(t)-1}\right)^2
\int_{a(t)}^1|\beta'-\beta'_1|^2
\left(\frac{\tau-1}{a-1}s+\frac{a-\tau}{a-1}\right)ds=\\
&=&\left(\frac{\tau}{a(t)}\right)
\int_0^{\tau}|\beta'-\beta'_1|^2(r)dr+\\
&&+\left(\frac{\tau-1}{a(t)-1}\right)
\int_{\tau}^1|\beta'-\beta'_1|^2(r)dr\leq\\
&\leq&\max\left(\frac{\tau}{a(t)},\frac{\tau-1}{a(t)-1}\right)
\int_0^1|\beta'-\beta'_1|^2(r)dr\leq\\
&\leq&M^2(t)||\beta-\beta_1||^2_{H^1}\;\;\forall t.
\end{eqnarray*}
where $M(t)$ is a continuous function s.t. $M(0)=1$.

In particular $\forall\eps>0$ there
exists $\tilde t>0$ s.t. for $t\leq\tilde t$
\begin{equation}
d(\h(\beta,t),\h(\beta_1,t))<
\left(1+\frac{\eps}{2}\right)||\beta-\beta_1||_{H^1}.
\end{equation}
So, if $\beta_1\in B(\beta,d')$, for all $\eps>0$ a $\tilde t$ exists s.t.
\begin{eqnarray*}
d(\h(\beta_1,t),\beta)&\leq&\\
&\leq&d(\h(\beta_1,t),\h(\beta,t))+d(\h(\beta,t),\beta)\leq\\
&\leq&\left(1+\frac{\eps}{2}\right)d'+\frac{\eps}{2}\,d'\\
&\leq&(1+\eps)d'\;\;\;\;\forall\,0\leq t\leq\tilde t,
\end{eqnarray*}
because $\h(\beta,t)$ is continuous in $t$. Notice that $\tilde t$ is
independent from $\beta_1$, so we have that, chosen $\beta$ and $d'$ s.t.
$B(\beta,d')\subset B(\gamma,d)$, then
for every $\eps>0$ there exists $\tilde t>0$ such that
\begin{equation}\label{dprimo}
\h(B(\beta,d'),t)\subset B(\beta,(1+\eps)d') \;\;\forall 0\leq t\leq\tilde t
\end{equation}

III) We have now to compound all these retraction. We follow an idea shown
by Corvellec, Degiovanni and Marzocchi in \cite[theorem 2.8]{CDM93}, and
we combine it with the compactness of $\Sigma_0$.

Take $d$ as in the first step. Then $\bigcup_\gamma B(\gamma,d/4)$ covers
$\Sigma_0$. By compactness we can choose
$$
\gamma_1,\cdots,\gamma_N\text{ s.t. }
\bigcup_{i=1}^N B\left(\gamma_i,\frac{d}{4}\right)\supset\Sigma_0.
$$
Set $B(\gamma_i,d/4)=B_i$, $R=\bigcup_i\overline{B}_i$ and
$\nu=\min_i\nu_{\gamma_i}$ and $\h_i=\h_{\gamma_i}$. Let
$$
\vartheta_i:H^1(I,M)\rightarrow [0,1]
$$
a partition of unity referred to $B_i$.

We want to define a sequence of continuous maps
\begin{displaymath}
\eta_h:R\times[0,\tilde t_h ]\rightarrow \Omega,
\end{displaymath}
for $h=1,\cdots,N$, defined as follows:
\begin{equation}
\eta_1(\beta,t)=\left\{
\begin{array}{ll}
\h_1(\beta,\vartheta_1t),& \beta\in\overline{B}_1\\
\beta, &\text{outside;}
\end{array}
\right.
\end{equation}
\begin{equation}
\eta_h(\beta,t)=\left\{
\begin{array}{ll}
\h_h(\eta_{h-1}(\beta,t),\vartheta_ht),& \beta\in\overline{B}_h\\
\eta_{h-1}(\beta,t),&\text{outside.}
\end{array}
\right.
\end{equation}
We want that, for all $h$,
\begin{enumerate}
\item $\eta_h(\beta,0)=\beta$;
\item $E(\eta_h(\beta,0))-E(\beta)\leq-\nu t \sum\limits_{i=1}^h\vartheta_i$;
\item $\forall i,\forall \eps\, \exists \tilde t_h$ s.t.
$\eta_{h-1}(\overline{B}_i,t)\subset B\big(\gamma_i,(1+\eps)^{h-1}d/4\big)$
if $0\leq t\leq \tilde t_h$.
\end{enumerate}

The proof of the first two condition is obvious. The last
condition, that assures the good definition of $\eta_h$, will be
proved by induction on $h$.

a) Case $h=1$:

If $B_i=B_1$ then $\eta_1(\beta,t)=\h_1(\beta,\vartheta t)$. Hence
there exists $\tilde t$ s.t., if $0\leq t\leq \tilde t$
\begin{equation}
d(\gamma_1,\eta_1(\beta,t))=d(\gamma_1,\h_1(\beta,\vartheta_1 t))<
(1+\eps)\frac{d}{4},
\end{equation}
in fact we know that there exists $\tilde t$ s.t.
\begin{equation}
d(\gamma_1,\h_1(\beta,t))<(1+\eps)\frac{d}{4},
\end{equation}
for all $0\leq t\leq \tilde t$, and $\vartheta_1\leq 1$, so
$\vartheta_1 t\leq t\leq \tilde t$.

If $B_1\cap B_i=\emptyset$, then
\begin{equation}
\eta_1(\overline{B}_i,t)=\overline{B}_i\;\;\forall t,
\end{equation}
so the proof is obvious.

Finally, if  $B_1\cap B_i\neq\emptyset$, we know that
\begin{displaymath}
\overline{B}\left(\gamma_i,\frac{d}{4}\right)\subset B(\gamma_1,d),
\end{displaymath}
so we can say that
$$
\eta_1(\overline{B}_i,t)=\h_1(\overline{B}_i,\vartheta_1t),
$$
hence we can repeat the above deduction. Taking the minimum of $\tilde t$
so found (they are a finite number) we can conclude.

b) Inductive step.

Let $\eta_{h-1}$ be s.t., given $\eps>0$, for all $i$ there
exists a $\tilde t_{h-1}$ for which
\begin{equation}
\eta_{h-1}(\overline{B}_i,t)\subset
B\left(\gamma_i,(1+\eps)^{h-1}\frac{d}{4}
\right)\;\;\forall 0\leq t\leq \tilde t_h.
\end{equation}
At first notice that we can choose $\eps$ s.t.
$$
\eta_{h-1}(\overline{B}_h,t)\subset B(\gamma_n,d),
$$
so $\eta_h$ is well defined.

Either if $B_i=B_h$ or if $B_h\cap B_i=\emptyset$ the proof is obvious.

Let $B_h\cap B_i\neq\emptyset$; if
$\beta\in B_i\smallsetminus \overline{B}_h$ then
$\eta_{h}(\beta,t)=\eta_{h-1}(\beta,t)$, so
\begin{equation}
d(\gamma_h,\eta_h(\beta,t))<(1+\eps)^{h-1}\frac{d}{4}<
(1+\eps)^{h}\frac{d}{4}.
\end{equation}
Otherwise, by inductive step
$$
\eta_{h-1}(\overline{B}_h,t)\subset
B\left(\gamma_h,(1+\eps)^{h-1}\frac{d}{4}
\right),
$$
and, by (\ref{dprimo}) we have that there exists a $\tilde t$ s.t
\begin{equation}
\h_h(\eta_{h-1}(\beta,t))\subset
B\left(\gamma_h,(1+\eps)^{h}\frac{\delta}{4}\right)
\end{equation}
if $\beta\in \overline{B}_h\cap B_i$ and $0\leq t\leq \tilde t$,
so the proof follows immediately. Because we have $N$ iterations,
we choose $\bar\eps$ s.t. $(1+\bar\eps)^N<2$, and we define
\begin{equation}
\bar t=\min_h\{t_{\bar \eps,h}\text{ previously found }\}.
\end{equation}
By compactness $\bar t\,>0$. Set
\begin{equation}
\eta_R=\eta_N,
\end{equation}
so we find a continuous map
\begin{equation}
\eta_R:R\times[0,\bar t\,]\rightarrow \Omega
\end{equation}
such that
\begin{itemize}
\item $\eta_R(\beta,0)=\beta$,
\item{$E(\eta_t(\beta,t))-E(\beta)\leq -\nu t
\sum\limits_{i=1}^N\vartheta_i=-\nu t$},
\end{itemize}
for every $\beta\in R$, $0\leq t\leq \bar t$.
\end{proof}

\begin{lemma}\label{retru}
For any $U\supset\Sigma_0$ there exist $\bar t,\nu\in\R^+$ and a
continuous functional
\begin{displaymath}
\eta_U:
\Omega_a^b\smallsetminus U\times [0,\bar t\,]
\rightarrow \Omega_a^b
\end{displaymath}
such that
\begin{itemize}
\item $\eta_U(\cdot,0)=\text{Id}$,
\item $E(\eta_U(\beta,t))-E(\beta)\leq-\nu t$,
\end{itemize}
for all $t\in[0,\bar t\,]$, for all $\beta \in \Omega_a^b\smallsetminus U$
\end{lemma}

\begin{proof}
We look for a pseudo gradient
vector field $F$, s.t., if $\eta_U$ is a solution of

\begin{equation}
\label{eta}
\left\{
\begin{array}l
\dot\eta_U(t,\gamma)=F(\gamma)\\
\eta_U(0,\cdot)=\text{Id}
\end{array}
\right.
\end{equation}
then
\begin{equation}
\label{stima}
\exists \nu>0 \text{ s.t } E(\eta_U(t,\gamma))-E(\gamma)<-\nu t.
\end{equation}

For every $S$ neighborhood of $\Sigma$, we have that $-\nabla E$ is
a good gradient field on $\Omega_a^b\smallsetminus S$, in fact $E$ is smooth
and satisfies the Palais Smale condition outside $S$, so, for
$\Omega_a^b\smallsetminus S$ does not contain
critical points of $E$, we know that there exists a $\nu_0\in\R^+$ s.t.
\begin{equation}
-||\nabla E||^2<-\nu_0;
\end{equation}
by integrating (\ref{eta}) with $F=-\nabla E$ we have that
\begin{equation}
E(\eta_U(\gamma,t))-E(\gamma)<-\nu_0 t
\end{equation}
if $\eta_U(\gamma,t)\subset
\Omega_a^b\smallsetminus S$ for all $t$.

Now let $S_1$ be a neighborhood of $S$ and let $U$ be a neighborhood
of $\Sigma_0$:
we look for a pseudo gradient vector
field on $S_1\smallsetminus U$.
Although $E$ is non smooth,
we can define $d E(\gamma)[w]$ for every $\gamma\in \Sigma\smallsetminus U$,
and for a suitable choice of $w$. It
is sufficient to take $w$ vector field along $\gamma$ with
\begin{displaymath}
\text{spt }w\subset \{s \text{ s.t. }\gamma(s)\neq v\}.
\end{displaymath}

There exists $\nu_1$ such that for every $\gamma\in\Sigma\smallsetminus U$
we can find $w_\gamma$ for which $d E(\gamma) [w_\gamma]
<-2\nu_1$. This is possible because we can find a partition
$0=s_0<\dots<s_k=1$ such that $\gamma(s_i)=v$ and $v\notin$Im
$\gamma|_{(s_i,s_{i+1})}$. Called
$\gamma_i=\gamma|_{(s_i,s_{i+1})}$ we can shorten it by a vector
field $w_i$ along $\gamma_i$, leaving its extremal point fixed, so
we obtain a vector field $w_\gamma$ along $\gamma$ with
\begin{displaymath}
\text{spt }w_\gamma\subset \{s \text{ s.t. }\gamma(s)\neq v\},
\end{displaymath}
and
\begin{equation}
d E(\gamma) [w_\gamma] <-2\nu_1,
\end{equation}
in fact for these variations the (P.S.) condition for energy holds. Moreover,
$\Sigma\smallsetminus U$ does not contain any stationary point for these
kind of variations.

Without loss of generality suppose now that exists a global chart
$(V,\phi)$, $0\in V\subset \R^n$ s.t. $\phi(0)=v$. The metric of
$M$, read on $V$, lead us to consider a matrix $(g_{ij}(x))_{ij}$
whose coefficients are discontinuous at 0; if $\gamma$ is a path
on $V$ we can compute is energy by taking

\begin{equation}
E(\gamma)=\int g_{ij}(\gamma)\gamma'_i\gamma'_j ds.
\end{equation}

For the sake of simplicity we suppose also that
\begin{displaymath}
g_{ij}(x)=g(x)\delta_{ij}(x)
\end{displaymath}
where $\delta_{ij}$ are the coefficient of Euclidean metric. The
general case does not present further difficulties.

Now we pass to coordinates $(V,\phi)$. Because $\gamma\in \Omega$, if
$||\gamma-\gamma_1||_{\Omega}<\eps$ then there exists $C\in\R^+$ s.t.
$||\gamma-\gamma_1||_{L^\infty}<C \eps$ by the Sobolev immersion,
so also $||\phi(\gamma)-\phi(\gamma_1)||_{L^\infty}<C \eps$.

In coordinates $d E (\gamma)[w] $ has the
following form:
\begin{equation}
d E (\gamma)[w]=\int g(\gamma)\gamma' w' ds + \int <\nabla g,w>
|\gamma'|^2 ds,
\end{equation}
where $w\in H^1(I,V)$.
Note that, even if $\nabla g$ does not exist everywhere, it is well defined
on spt $w$.

We have proved that for every $\gamma \in \Sigma\smallsetminus U$
exists $w_\gamma$ s.t.
$d E(\gamma) [w_\gamma] <-2\nu_1$; obviously we can prove the same for every
$\gamma \in S_1\smallsetminus U$. Given
$\gamma \in S_1\smallsetminus U$ and $w_\gamma$ as
above, it exists
a neighborhood $V_\gamma$ of $\gamma$ s.t.
\begin{equation}
\label{stima2}
\forall \gamma_1\in V_\gamma\,\,
d E(\gamma_1) [w_\gamma] <-\nu_1.
\end{equation}

Let $||\gamma-\gamma_1||_{H^1}<\eps$, then
\begin{eqnarray*}
\int g(\gamma)\gamma' w'_\gamma &-&
\int g(\gamma_1)\gamma'_1 w'_\gamma \leq\\
\leq\int g(\gamma)\left(\gamma'-\gamma'_1\right)w'_\gamma&+&
\int \left(g(\gamma)-g(\gamma_1)\right)\gamma'_1w'_\gamma\leq\\
\leq\sup\limits_{t\in\text{spt }w_\gamma}g(\gamma)
||\gamma'-\gamma'_1||_{L^2}||w'||_{L^2}&+&
\sup\limits_{t\in\text{spt
}w_\gamma}\left[g(\gamma)-g(\gamma_1)\right]
||\gamma'_1||_{L^2}||w'||_{L^2}\leq\\
&\leq& \text{Const }\cdot\eps,
\end{eqnarray*}
in fact $g(\gamma)\in C^\infty(\text{spt } w_\gamma)$, so $\sup g(\gamma)$
is bounded; furthermore,because $||\gamma-~\gamma_1||_{L^\infty}<C\cdot\eps$,
$\sup\left[g(\gamma)-g(\gamma_1)\right]\leq~C\cdot~\eps$.

In the same way
\begin{eqnarray*}
\int<\nabla g(\gamma),w_\gamma>|\gamma'|^2&-&
\int<\nabla g(\gamma_1),w_\gamma>|\gamma'_1|^2\leq\\
\int<\nabla g(\gamma),w_\gamma>(|\gamma'|^2-|\gamma'_1|^2)&+&
\int<\nabla g(\gamma)-\nabla g(\gamma_1),w_\gamma>|\gamma'_1|^2\leq\\
&\leq& \text{Const }\cdot\eps.
\end{eqnarray*}

So $d E(\gamma)[w_\gamma]-d E(\gamma_1)[w_\gamma]\leq C\cdot\eps$: we can
choose a neighborhood $V_\gamma$, for all $\gamma\in S_1\smallsetminus U$,
s.t.
\begin{equation}
d E(\gamma_1)[w_\gamma]<-\nu_1\,\,\forall \gamma_1 \in V_\gamma.
\end{equation}
The sets $V_\gamma$ covers the whole $S_1\smallsetminus U$.
Let $V_{\gamma_i}$ be a locally
finite refinement of $V_\gamma$.
Let $\beta_i$ be a partition of the unity associated to $V_{\gamma_i}$. Then
\begin{equation}
F_1=\sum\beta_i w_{\gamma_i}
\end{equation}
is a pseudo-gradient vector field on $S_1\smallsetminus U$
(for the details of such a
construction see \cite{Rab74}). Now let $\alpha_j$ be a partition
of the unity associated to
$S_1\smallsetminus U,\Omega_a^b\smallsetminus S$, then
\begin{equation}
F=\alpha_1 F_1 -\alpha_2\nabla E
\end{equation}
is the vector field we looked for, in fact we can find $\eta_U$
because $F$ is a Lipschitz
vector field by definition. Even if
$E$ isn't smooth, we can differentiate it along the direction of F,
so
\begin{eqnarray*}
E(\eta_U(\gamma,t))-E(\gamma)&=&
\int\limits_0^t \frac{d}{d\tau}E(\eta_U(\gamma,\tau)) d\tau=\\
=\int\limits_0^t dE\left(\eta_U(\gamma,\tau)\right)
\left[\dot\eta_U(\gamma,\tau)\right]d\tau&=&
\int\limits_0^t dE\left(\eta_U(\gamma,\tau)\right)[F].
\end{eqnarray*}
Let $\nu=\min(\nu_0,\nu_1)$. Then
\begin{eqnarray*}
E(\eta_U(\gamma,t))&-&E(\gamma)=\\
=\int_0^t\alpha_1dE\left(\eta_U(\gamma,\tau)\right)[F_1]&-&
\alpha_2||\nabla E\left(\eta_U(\gamma,\tau)\right)||^2=\\
=\int_0^t\alpha_1\sum\beta_idE\left(\eta_U(\gamma,\tau)\right)[w_{\gamma_i}]
&-&\alpha_2||\nabla E\left(\eta_U(\gamma,\tau)\right)||^2\leq\\
\leq\int_0^t-\alpha_1\sum\beta_i\nu_1&-&\alpha_2\nu_0\\
\leq-\int_0^t\nu&\leq&-\nu t.
\end{eqnarray*}
\end{proof}

From lemma \ref{retrsigma} and lemma \ref{retru} we get the
following result.

\begin{teo}\label{gamma0}
Let $M$ be a conical manifold with only a vertex $v$,
and consider the special closed geodesic
$\gamma_0$ for which there exists an unique $\sigma$
s.t. $\gamma_0(\sigma)=v$. Set $E(\gamma_0)=c_0$.
Suppose that there exist $a,b\in\R$, $c_0<a<b$ s.t. $\Omega^b$
contains only the geodesics
$\gamma_0$.

Then $\Omega^b\simeq\Omega^a$.
\end{teo}
\begin{proof}
Given $R$ as in lemma \ref{retrsigma},
we choose $U$ and $V$ neighborhoods of $\Sigma_0$ s.t.
$$
\Sigma_0\subsetneq U\subsetneq V\subsetneq R.
$$
we know that, for such an $U$, there exists a retraction $\eta_U$
defined as in Lemma \ref{retru}. For the sake of simplicity we
will suppose that $\eta_U$ and $\eta_R$ (see Lemma
\ref{retrsigma}) are defined for $0\leq t \leq 1$ and that $\nu$
is the same for both of them. Let
$\theta_1:\Omega^b\rightarrow[0,1]$ a continuous map s.t.
\begin{eqnarray*}
\theta_{1}|_U&\equiv&0\\
\theta_{1}|_{\Omega^b\smallsetminus V}&\equiv&1.
\end{eqnarray*}
Then we define a continuous map
\begin{equation}
\mu_1:\Omega^b\times[0,1]\rightarrow\Omega^b,
\end{equation}
\begin{equation}
\mu_1(\beta,t)=\eta_U(\beta,\theta_1(\beta)t);
\end{equation}
we know that $E(\mu_1(\beta,t))-E(\beta)\leq-\nu t\theta_1(\beta)$, so
\begin{displaymath}
\mu_1(\Omega^b,1)\subset V\cup\Omega^{b-\nu},
\end{displaymath}
in fact if $\mu_1(\beta,t)\notin V$ for all $t$, then
$E(\mu_1(\beta,t))-E(\beta)\leq-\nu t$, so
$\mu_1(\beta,1)\in\Omega^{b-\nu}$.

By $\mu_1$ we have retracted $\Omega^b$ on
$\Omega^{b-\nu}\cup V$; now we define a continuous map
$\theta_2:\Omega^b\rightarrow[0,1]$ s.t.
\begin{eqnarray*}
\theta_{1}|_{\Omega^b_{b-\nu/2}}&\equiv&1,\\
\theta_{1}|_{\Omega^{b-\nu}}&\equiv&0.
\end{eqnarray*}
Then set
\begin{equation}
\mu_2:V\cup\Omega^{b-\nu}\times[0,1]\rightarrow\Omega^b
\end{equation}
\begin{equation}
\mu_2(\beta,t)=\eta_R(\beta,\theta_2(\beta)t);
\end{equation}
$\mu_2$ is a continuous map that retracts $V\cup\Omega^{b-\nu}$ on
$\Omega^{b-\nu/2}$. By iterating this algorithm we can retract
continuously $\Omega^b$ on $\Omega^a$.
\end{proof}

Now we can prove the deformation lemma.

\begin{proof}[Proof of Theorem \ref{maindeflemma}]
Let $\{\gamma_i\}_{i=1,\cdots,N}$ be the set of geodesics in $\Omega^b$.
We start defining some special subset of $\Omega_a^b$,
as in (\ref{sigma}) and (\ref{sigma0}); let
\begin{equation}
\Sigma=\{\gamma\in \Omega_a^b,
\text{ s.t. Im}\gamma\cap V\neq\emptyset\}
\end{equation}
(we recall that $V$ is the set of vertexes); for $i=1,\cdots N$, set
\begin{equation}
\Sigma_i=\{\gamma\in\Sigma\text{ s.t. }\gamma=\gamma_i
\text{ up to affine reparametrization }\}.
\end{equation}
We note that for $i\neq j$ then $\Sigma_i\cap\Sigma_j=\emptyset$, because the
geodesics are different. For
these $\Sigma_i$ we can find a retraction $\eta_{\Sigma}$ as in
lemma \ref{retrsigma}: indeed, for every
$U\supset\bigcup\limits_{i=0}^N\Sigma_i$ there
exists a retraction $\eta_U$
on $\Omega_a^b\smallsetminus U$ in analogy
with lemma \ref{retru}. Finally, we compound these two maps
$\eta_{\Sigma}$ and $\eta_U$ following the proof of
theorem \ref{gamma0} and we conclude.
\end{proof}

\begin{teo}[Second deformation lemma]\label{deflemma2}
Let $M$ be a
conical manifold, $p\in M$. Suppose that there exists $c\in\R$ s.t.
$\Omega^c$ contains only a finite number of geodesics and that
there exists $a,b\in R$, $a<b<c$ s.t. the strip $[a,b)$
contains only regular values of $E$. Set
$Z$ the set of geodesics and $Z_b=Z\cap E^{-1}(b)$,
then there exists a neighborhood $U$ of
$Z_b$ s.t.
\begin{displaymath}
\Omega^b\smallsetminus U\simeq\Omega^a.
\end{displaymath}
\end{teo}
\begin{proof}
We can prove this corollary following the lines of Theorem \ref{maindeflemma}.
\end{proof}

As previously said, Lemma \ref{retrsigma}, which is crucial for this work,
is based on a generalization of \cite[theorem 2.8]{CDM93}. Indeed, using a
slight modification of the weak slope tool, this result and the
deformation lemmas can be reformulated in a more general context. This
theoretic frame is briefly discussed in the appendix.
\section{Category theory}

First, we recall some well known results relative to
the Lusternik and Schnirelmann
category. This theory was presented in \cite{LS34} in a finite
dimensional framework, then generalized to Banach manifold by R. Palais
\cite{Pal66a}.
\begin{defin}
Let $X$ be a topological space, $A\subset X$. If $A\neq\emptyset$ we say
that
\begin{displaymath}
\cat A=\cat_XA=k\;\;\text{ iff}
\end{displaymath}
$k$ is the least integer for which there are $F_1,\cdots,F_k$
closed contractible subsets of $X$ s.t. $\bigcup_kF_k$ covers $A$.

We define also
\begin{displaymath}
\cat\emptyset=\cat_X\emptyset=0.
\end{displaymath}
\end{defin}

\begin{teo}\label{maincat}
Let $X$ be a topological space. Then
\begin{enumerate}
\item{if $A\subset B\subset X$ then $\cat_XA\leq\cat_XB$;}
\item{if $A,B\subset X$ then $\cat_XA\cup B\leq\cat_XA+\cat_XB$;}
\item{if $A,B\subset X$, $A$ closed, and there is
$\eta\in C([0,1]\times A,X)$ s.t.
\begin{eqnarray*}
&&B=\eta(1,A);\\
&&\eta(0,u)=u\;\;\forall u\in A,
\end{eqnarray*}
then $\cat_XA\leq\cat_XB$}
\item{if $Y$ is a topological space, $y\in Y$, then
$\cat_{X+Y}(A\times\{y\})+\cat_XA$}.
\end{enumerate}
\end{teo}
\begin{proof}
The points 1,2 and 4 are trivial. We have only to prove 3.

By hypothesis, we can find $F_1,\cdots,F_k$ s.t.
$B\subset F_1\cup\cdots\cup F_k$. Set
\begin{displaymath}
C_i=\left\{
u\in A\text{ s.t. }\eta(1,u)\in F_i
\right\}.
\end{displaymath}
Obviously, $C_i$ are closed and contractible. Since
$C_1\cup\cdots\cup C_k$ covers $A$ we obtain the thesis.
\end{proof}

By Theorem \ref{deflemma2} we are able to reconstruct the category theory
for the energy functional defined on a conical manifold.
\begin{lemma}
Let $M$ be a conical manifold, $p\in M$. Suppose that there exists
$\bar c\in\R$ s.t.
$\Omega^{\bar c}$ contains only a finite number of geodesics.
Let $c<\bar c$ a critical level for $E$.
Then, set $U$ a neighborhood of $Z_c$ there exists $\eps>0$ s.t.
\begin{equation}
\cat\Omega^{c+\eps}\leq\cat\Omega^{c-\eps}+\cat U.
\end{equation}
\end{lemma}
\begin{proof}
We know, by the second deformation lemma, that
$\Omega^{c-\eps}$ is a deformation
retract of $\Omega^{c+\eps}\smallsetminus U$: applying Theorem \ref{maincat}
we obtain
\begin{equation}
\cat \Omega^{c+\eps}\leq\cat \Omega^{c+\eps}\smallsetminus U+\cat U\leq
\cat \Omega^{c-\eps}+\cat U.
\end{equation}
\end{proof}

\begin{teo}\label{numpunticrit}
Let $M$ be a conical manifold, let $p\in M$ and let $a<b\in\R$. Then
$\Omega_a^b$ contains at least $\cat \Omega^b-\cat \Omega^a$ geodesics.
\end{teo}
\begin{proof}
We suppose that there is a finite number of critical levels in $[a,b]$
(otherwise there is nothing to prove). Set
$a\leq c_0<c_1<\cdots<c_k\leq b$ these critical levels, and set,
for all $i$, $U_i$ a neighborhood
of $Z_{c_i}$. We know that there exists an $\eps$ s.t. for all $i$
\begin{equation}\label{catmainteo}
\cat \Omega^{c_i+\eps}\leq\cat \Omega^{c_i-\eps}+\cat U_i.
\end{equation}
By iterating (\ref{catmainteo}), and using the deformation lemma,
we obtain
\begin{eqnarray*}
\cat \Omega^{c_k+\eps}&\leq&\cat \Omega^{c_k-\eps}+\cat U_k\leq
\cat \Omega^{c_{k-1}+\eps}+\cat U_k\leq\\
&\leq&\cat \Omega^{c_{k-2}+\eps}+\cat U_{k-1}+\cat U_k\leq\cdots\leq\\
&\leq&\cat \Omega^{c_0-\eps}+\sum_{i=0}^k\cat U_i.
\end{eqnarray*}
Because $\cat \Omega^b\leq \cat \Omega^{c_k+\eps}$ and
$\cat \Omega^{c_0-\eps}\leq\cat \Omega^a$ we have that
\begin{equation}
\cat \Omega^b-\cat \Omega^a\leq\sum_{i=0}^k\cat U_i.
\end{equation}
Suppose now that there are a finite number of geodesics for any
critical level. Because every point has a contractible neighborhood,
we can choose $U_i$ s.t.
\begin{equation}
\cat U_i\leq \#Z_{c_i},
\end{equation}
thus
\begin{equation}
\cat \Omega^b-\cat \Omega^a\leq \sum_i\#Z_{c_i}.
\end{equation}
\end{proof}
From theorem \ref{numpunticrit} the main result of this paper follows.
\begin{cor}\label{catmaincor}
Let $M$ be a conical manifold, $p\in M$. Then there are at least
$\cat\Omega$ geodesics.
\end{cor}
\begin{proof}
If there is an infinite number of geodesics, there is nothing to prove.

Otherwise,
we can apply the previous theorem and we conclude by a limiting process.
(consider that $\Omega^{-1}=\emptyset$ and that $\Omega^b\simeq\Omega$ for
$b>>1$).
\end{proof}
\section{An application}\label{sezappl}

We show a topological lemma necessary to provide some applications.

Let $X$ a smooth submanifold of $\R^{n}$. Given $g\in
L^\infty(X,\R^+)$, set
\begin{displaymath}
E(\gamma)=\int_0^1g(\gamma(s))|\gamma'|^2ds.
\end{displaymath}
We set
\begin{eqnarray*}
G(I,X)&=&\{\gamma\in C^0(I,X)\;:\:E(\gamma)
\text{ is well defined and finite}\};\\
G(S^1,X)&=&\{\gamma\in C^0(S^1,X)\;:\:E(\gamma)
\text{ is well defined and finite}\}.
\end{eqnarray*}
Obviously we have that
\begin{eqnarray}
H^1(I,X)\subset &G(I,X)&\subset C^0(I,X);\\
H^1(S^1,X)\subset &G(S^1,X)&\subset C^0(S^1,X).
\end{eqnarray}
We recall that
\begin{eqnarray*}
&&\Omega=\Omega_p X=\{\gamma\in H^1([0,1],X)\;:\:\gamma(0)=
\gamma(1)=p\};\\
&&\Omega^\infty=\Omega_p^\infty X=\{\gamma\in
C^0(S^1,X)\;:\:\gamma(0)= \gamma(1)=p\},
\end{eqnarray*}
as previously defined.
We define also the free loop space on $X$ as
\begin{eqnarray*}
&&\Lambda=\Lambda X=\{\gamma\in H^1(S^1,X)\};\\
&&\Lambda^\infty=\Lambda^\infty X=\{\gamma\in C^0(S^1,X)\}.
\end{eqnarray*}
In analogous way we set $G$ (resp. $G_p$) the subspace of
$\Lambda^\infty$ (resp. $\Omega^\infty$) in which $E$ is well
defined and finite, according with previous definitions.
These definitions allow us to formulate the following lemma.

\begin{lemma}\label{majer}
Let $X,g$ and $E(\cdot)$ be as above.
Then
\begin{eqnarray}
\cat_GG&\geq&\cat_{\Lambda^\infty}\Lambda^\infty;\label{majer1}\\
\cat_{G_p}{G_p}&\geq&\cat_{\Omega^\infty}\Omega^\infty.\label{majer2}
\end{eqnarray}
In particular, if $X$ is a connected and non contractible manifold
then
\begin{equation}\label{majer3}
\cat G_p=\cat{\Omega^\infty}=\infty.
\end{equation}
\end{lemma}
\begin{proof}
We show only (\ref{majer1}), then (\ref{majer2}) follows in the
same way. Because $X$ is a smooth manifold, it is well known that
there is an homotopic equivalence between $\Lambda^\infty$ and
$\Lambda$ (see, e.g. \cite[Th 1.2.10]{Kli78}). Then
\begin{displaymath}
\cat_{\Lambda^\infty} \Lambda=\cat_{\Lambda^\infty}\Lambda^\infty;
\end{displaymath}
now, because $\Lambda\subset G\subset \Lambda^\infty$, we have
\begin{equation}
\cat_GG\geq\cat_{\Lambda^\infty}\Lambda=\cat_{\Lambda^\infty}\Lambda^\infty,
\end{equation}
that proves (\ref{majer2})

Formula (\ref{majer3}) is a standard result and can be found, for
example in \cite[Corollary 1.2]{FH91}
\end{proof}

By this result we can compute $\cat\Omega$ in some concrete
case, as shown in the next example.
\begin{ex}
Let $M\subset \R^n$ a compact conical manifold, $V$
the set of its vertexes. Suppose that there exists a compact smooth manifold
$X\subset\R^k$ and an homeomorphism $\psi:M\rightarrow X$ s.t.
\begin{displaymath}
\psi_{|_{M\smallsetminus V}}\in C^\infty(M\smallsetminus V,X),
\end{displaymath}
then there exists $g^*$ an induced metric on $X$ defined by
\begin{equation}
g^*_p(v,w):=
\left\{
\begin{array}{ll}
g_{\psi^{-1}(p)}\left(d\psi^{-1}(v),(d\psi^{-1}(w)\right)&
\text{on }X\smallsetminus \psi(V);\\
0&\text{otherwise}.
\end{array}
\right.
\end{equation}
If $|d\psi^{-1}|\in L^\infty(X)$, we have that $g^*$ is bounded
with respect to the Euclidean metric of $X$ and that
\begin{equation}
\Omega_{\psi^{-1}(p)}(M)=G_p(X):=\left\{\gamma\in C^0([0,1],X),
\int_0^1g^*_{\gamma}|\gamma'|^2<\infty\right\}.
\end{equation}
In this case, we can apply lemma \ref{majer} to compute the category of based
(or free) loop space of $M$.
\end{ex}
We also state an immersion theorem that is, in some sense, the converse of
previous example.
\begin{teo}[Nash immersion for conical manifolds]\label{nash}
Let $X$ a smooth manifold and let $g$ a continuous non negative
and bounded bilinear tensor s.t. there exist $V$ a finite set of
points and $g$ is smooth and positive defined on $X\smallsetminus
V$. Then
\begin{enumerate}
\item{If $V=\{x\}$, then, for $N$ sufficiently large,
there exists $M\subset\R^N$ a conical manifold
and a continuous map
\begin{displaymath}
\psi:X\rightarrow M
\end{displaymath}
s.t $\psi_{|_{X\smallsetminus V}}$ is a $C^\infty$ isometry.
}
\item{If $V=\{x_1,\cdots,x_k\}$, for every $x_i$ it exists $\rho_i>0$
s.t. $B(x_i,\rho_i)$ is isometric (in the sense above specified) to some
conical manifold $M_i\subset\R^N$ }.
\end{enumerate}
\end{teo}
\begin{proof}
We start proving 1. By hypothesis, $(X\smallsetminus V,g)$ is a
Riemannian manifold, so, by Nash theorem \cite{Na56}, it can be
embedded in $\R^N$, for $N$ sufficiently large. Let
$\psi:X\smallsetminus V\rightarrow M$ be this embedding.

We can continuously extend $\psi$ to the whole $X$. In fact, let
$\{x_n\}_n$ be a Cauchy sequence converging to $x$; because $g$ is
bounded, then $\{\psi(x_n)\}_n$ is a Cauchy sequence in $\R^N$, so
there exists $y\in\R^N$ s.t. $\lim\psi(x_n)=y$. Set $\psi(x):=y$:
obviously we have that
\begin{equation}
\psi(B_X(x,\rho))\subset B_{\R^N}(y,r),
\end{equation}
and $r\stackrel{\rho\rightarrow0}{\longrightarrow}0$, so $\psi$ is
continuous at $x$.

Then, set $M:=\psi(X)$, we have that $M$ is a conical manifold
with vertex $y$, isometric to $X$.

To proof 2, it is sufficient to choose $\rho_i$ s.t.
$B(x_i,\rho_i)$ are all disjoint. Then we apply the previous
result with $X=B(x_i,\rho_i)$.
\end{proof}
By this result, we formulate a result which will be useful in the next of this paper.
\begin{teo}
In the above hypothesis, we have that
\begin{displaymath}
\text{number of geodesics in }X\geq\cat G
\end{displaymath}
\end{teo}
\begin{proof}
If $X$ has an unique vertex, it is isometric to a conical manifold $M$.
Then, by applying lemma \ref{majer}, we obtain the proof.
If the manifold $X$ has several vertexes, we are in the case 2
of previous theorem.

Anyway, by the local isometries, we can prove an analogous of deformation lemma for
geodesics in $X$. Also an analogous of theorem \ref{numpunticrit} follows. This, paired with lemma
\ref{majer} gives us the proof.
\end{proof}
\subsection{Brachistochrones}\label{sezbrac}
In this section we want to study the brachistocrones problem. A
brachistochrone is a curve $\gamma$ which minimizes the time of
transit for a particle moving from a point $p$
towards a point $q$. We study this problem on $(S^n,<,>)$ an
Euclidean sphere embedded in $\R^{n+1}$. We suppose that the particle
moves in the presence of a potential $U:S^n\rightarrow\R$ without friction. Also, we are interested
to any curve stationary for the time of transit functional.

Be $p,q\in S^n$, $E\in\R^+$ the energy of the particle, $U\in
C^\infty(S^n,\R)$ the given potential. It is well known that, if
there exist $c_1,c_2\in\R$ s.t.
\begin{equation}
-\infty<c_1<U(\cdot)<c_2<E,
\end{equation}
then this problem is equivalent to the geodesic problem for the
Riemannian manifold
\begin{displaymath}
\left(
S^n,g:=g_x=\frac{<,>}{E-U(x)}
\right),
\end{displaymath}
and that the metric $g$ is equivalent to the Euclidean metric on $S^n$, so
the problem has always a solution.
Furthermore, it is also well known that a solution exists even if
the upper bound on $U(x)$ does not exists.

In this section we want to study the
problem for a given potential
\begin{displaymath}
U\in C^\infty(S^n\smallsetminus V,\R)
\end{displaymath}
where $V=\{x_1,\cdots x_k\}$ a finite set of points on the sphere, and
\begin{displaymath}
U(x)\stackrel{x\rightarrow x_i}{\longrightarrow}-\infty
\end{displaymath}
As we will see in the next section, potential in $S^n$ with these kind of singularities
may appear from non singular potential defined in $\R^n$.

For the sake of simplicity, we suppose that there exist $c>0$ for which $E>c>U(\cdot)$.

We define a metric on $S^n$ by
\begin{equation}
g:=g_x=
\left\{
\begin{array}{ll}
\frac{<,>}{E-U(x)}&\text{on }S^n\smallsetminus V;\\
0&\text{otherwise},
\end{array}
\right.
\end{equation}
and we look for $g_x$-geodesics between two given points $p,q\in
S^n$. Set, as usual
\begin{eqnarray*}
G(p)&=&\left\{\gamma\in C^0([0,1],S^n),\;\gamma(0)=\gamma(1)=p,
\frac{1}{2}\int_0^1g(\gamma',\gamma')<\infty
\right\};\\
G(p,q)&=&\left\{\gamma\in C^0([0,1],S^n),\;\gamma(0)=p,\gamma(1)=q,
\frac{1}{2}\int_0^1g(\gamma',\gamma')<\infty
\right\},
\end{eqnarray*}
we know that $g$ satisfies the hypothesis of lemma \ref{majer}, so
\begin{equation}
\cat G(p)=\infty.
\end{equation}
It's easy to prove that there is an homotopy equivalence between
$G(p)$ and $G(p,q)$, in fact, for any given couple of points $p,q$,
there exists a continuous curve $\gamma$ which joins them, with
$E(\gamma)<\infty$ (because $g$ is bounded). Then there is a map
\begin{eqnarray*}
i&:&G(p)\rightarrow G(p,q);\\
&&\beta\mapsto\beta+\gamma,
\end{eqnarray*}
where $\beta+\gamma$ is the usual composition of paths.

Of course there exists the inverse map
\begin{eqnarray*}
i^{-1}&:&G(p,q)\rightarrow G(p);\\
&&\beta\mapsto\beta +(-\gamma).
\end{eqnarray*}
and $i^{-1}\circ i$ is homotopic equivalent to $1_{G(p)}$.

By the above consideration and by Nash theorem we have that
\begin{displaymath}
\infty=\cat G(p)=\cat G(p,q)=\text{number of geodesics between }p\text{ and }q,
\end{displaymath}
thus we can count the number of brachistochrones on the sphere in presence of
our potential $U$.
\subsubsection{Brachistocrones in $\R^n$}
A more interesting application is the study of the same brachistochrone
problem in $\R^n$ (indeed this was the very beginning of our research).
Let $U\in C^\infty(\R^n,\R)$ and $E>0$ s.t.
\begin{itemize}
\item{$E>U(x)$,}
\item{$-U(x)=O(|x|^\alpha)$ when $|x|>>1$, for
some $\alpha>0$.}
\end{itemize} We are looking for brachistocrones
joining two given points $p,q\in R^n$ in presence of potential
$U(x)$. As above we look for geodesics in
\begin{equation}
\left(
R^n,g_x:=\frac{<,>}{E-U(x)}
\right),
\end{equation}
where $\frac{1}{E-U(x)}\in C^\infty(\R^n,\R\smallsetminus\{0\})
\cap L^\infty(\R^n)$.

We can map $\R^n$ in $S^n\subset\R^n+1$ by the stereographic map $\pi$.
The inverse map is
\begin{equation}
\pi^{-1}:
\left\{
\begin{array}{ccc}
S^n\smallsetminus N\subset\R^{n+1} &\rightarrow& \R^n\\
\left(
\begin{array}{c}
y_1\\
\cdots\\
y_{n+1}
\end{array}
\right)
&\mapsto&
\left(
\begin{array}{c}
x_1\\
\cdots\\
x_{n}
\end{array}
\right)
=
\left(
\begin{array}{c}
\frac{y_1}{1-y_{n+1}}\\
\cdots\\
\frac{y_{n}}{1-y_{n+1}}
\end{array}
\right)
\end{array}
\right.,
\end{equation}
where $N$ is the north pole of $S^n$. As usual we can induce a metric $g^*$
on $S^n$ defined by
\begin{equation}
g^*(y)(v,w)=
\left\{
\begin{array}{cl}
g_{\pi^{-1}(y)}
\left(
d\pi^{-1}v,d\pi^{-1}(w)
\right)
&\text{on }S^n\smallsetminus N,\\
0&y=N.
\end{array}
\right.
\end{equation}
It'easy to see that
\begin{equation}
|d\pi^{-1}|=O\left(\frac{1}{\sqrt{1-y_{n+1}}}\right),
\end{equation}
that, read on $\R^n$, becomes
\begin{equation}
|d\pi^{-1}|=O\left(\frac{1}{|x|}\right).
\end{equation}
So, if $\alpha>2$, then $g^*$ is bounded with respect to the
Euclidean metric on $S^n$, and we can apply lemma \ref{majer}.

Furthermore, by Nash embedding, there is an isometry with a
compact conical manifold, so we can easily state that there is an
infinite number of brachistocrones joining $p$ and $q$, although
we cannot say if they are bounded in $\R^n$, and so physical
meaningful.

\appendix

\section{The theoretic frame}

As said, our deformation lemmas
(Lemma \ref{maindeflemma} and Lemma \ref{deflemma2})
are obtained modifying a weak slope theory resutlt. In this section
we present the {\em k-slope},
a generalization of the weak slope, which allows us to reformulate
the main results of this paper in a more general framework.

We start recalling the definition of weak slope.
\begin{defin}
Let $(X,d)$ be a metric space and let $f:X\rightarrow\R$ be a
continuous functional. The {\em weak slope} of $f$ at $u\in
X$ (noted $|df|(u)$) is the supremum of $\sigma$'s in
$[0,+\infty)$ s.t. $\exists\delta>0$ and
$\h:B(u,\delta)\times[0,\delta]\rightarrow X$ continuous with
\begin{equation}\label{cap2ws1}
d(\h(v,t),v)\leq t
\end{equation}
\begin{equation}\label{cap2ws2}
f(\h(v,t))-f(v)\leq-\sigma t
\end{equation}
for every $v\in B(u,\delta)$, $t\in [0,\delta]$.
\end{defin}

Due to (\ref{cap2ws1}) we can prove a deformation property for
continuous functio\-nals (\cite[theorem 2.8]{CDM93}): this
inequality allows us to compound the local maps $\h$ finding a global
retraction.

Unhappily, these tools are not completely useful for our purposes.
In particular we was not able to prove an estimate like
(\ref{cap2ws1}).
In our work we override these difficulties using the compactness
of sets $\Sigma_i$ and compounding explicitly all the local retractions.
This method has a generalization that we present here.

\subsection{The $k$-slope}

We define an extension
of weak slope which will be called {\em $k$-slope}.

\begin{defin}\label{cap2ksdef}
Let $(X,d)$ be a metric space. Let $f:X\rightarrow \R$ be a continuous
functional and  let $u\in X$. We define the {\em k-slope} of $f$
at $u\in X$ (noted $|d_kf|(u)$) as the supremum of
$\sigma\in[0,\infty)$ s.t. exist $\delta>0$,
$k_u:[0,\delta]\rightarrow\R^+$ continuous, $k_u(0)=0$, and a
continuous map
$$
\h:B(u,\delta)\times[0,\delta]\rightarrow X
$$
which satisfies
\begin{eqnarray}
d(\h(v,t),v) &\leq&k_u(t)\label{cap2ksdef1}\\
f(\h(v,t))-f(v)&\leq&-\sigma t\label{cap2ksdef2}
\end{eqnarray}
for all $v\in B(u,\delta)$, for all $t\in[0,\delta]$
\end{defin}

In analogy with the weak slope theory we can prove the following
property.
\begin{prop}\label{cap2sci}
If $f$ is continuous, $|d_kf|$ is lower semi-continuous.
\end{prop}
\begin{proof}
If $|d_kf|(u)=0$ the proof is obvious. Otherwise, for any
$0<\sigma<|d_kf|(u)$ there exist $\delta$ and
$\h:B(u,\delta)\times[0,\delta]\rightarrow X$ as in definition
\ref{cap2ksdef}. Let $u_h\rightarrow u$. Definitively we have $u_h\in
B(u,\frac{\delta}{2})$, so we can take the restriction of $\h$ to
$B(u_h,\frac{\delta}{2})\times[0,\frac{\delta}{2}]$, to have
$|d_kf(u_h)|\geq \sigma$. This completes the proof.
\end{proof}
Obviously we say that $u\in X$ is a {\em critical point} if
$|d_kf|(u)=0$.
\subsection{The deformation lemma}

We are able now to formulate the wanted deformation property.
\begin{teo}\label{cap2ksmain}
Let $(X,d)$ be a metric space, and $f:X\rightarrow \R$ a conti\-nuous
functional. Suppose that exists $\sigma\in\R^+$ s.t.
$|d_kf|(u)\geq \sigma$ for all $u\in X$. Let $C\subset X$ be a compact
subspace such that
\begin{equation}\label{cap2kustima}
k_u(t)\leq t \hbox{ }\hbox{ } \forall u\in X\smallsetminus C.
\end{equation}

Then it exists a $\tau\in \R^+$ and a continuous function
$\mu:X\times[0,\tau]\rightarrow X$ s.t.
\begin{equation}
\mu(u,0)=u  \hbox{ }\hbox{ } \forall u\in X,
\end{equation}
\begin{equation}
f(\mu(u,t))-f(u)\leq-\sigma t \hbox{ }\hbox{ } \forall u\in X,t\in
[0,\tau].
\end{equation}
\end{teo}

Before proving \ref{cap2ksmain}, we prove two deformation lemmas for $C$ and
$X\smallsetminus C$ analogues to lemma \ref{retrsigma} and
lemma \ref{retru}. To conclude the proof we will attach the
retractions found.

We recall a topological lemma by John Milnor useful for the next results.

\begin{lemma}[Milnor's lemma]
Let $\{U_\alpha\}_{\alpha\in A}$ be an open cover of a paracompact space $X$.
There is a locally finite open cover $V_{j,\lambda}$ refining
$\{U_\alpha\}$ s.t. $V_{j,\lambda}\cap V_{j,\mu}=\emptyset$
if $\lambda\neq\mu$.
\end{lemma}
\begin{proof}
For the proof we refer to \cite[Lemma 2.4]{Pal66b}. Here we report only how to
construct the open cover
$\{V_{i,\lambda}\}_{i,\lambda}$.

By an initial refinement
we can take  $\{U_\alpha\}$ locally finite. Then, let $\Lambda_j$
be the set of $(j+1)$-ples
$\lambda=\{\alpha_0,\dots,\alpha_j\}$ of elements in $A$.
Let $\{\varphi_\alpha\}_\alpha$
be a partition of unity with $\text{spt}\varphi_\alpha\subset U_\alpha$; for
$\lambda \in \Lambda_j$ let
\begin{displaymath}
V_{j,\lambda}=\left\{
x\in X\;|\;\varphi_\alpha>0 \text{ if }\alpha\in\lambda\text{ and }
\varphi_\gamma<\varphi_\alpha \text{ if }\alpha\in\lambda\,,
\gamma\notin\lambda
\right\},
\end{displaymath}
so we have found our locally finite open cover $V_{j,\lambda}$.
\end{proof}

With this lemma, we prove the deformation results.

\begin{lemma}[deformation lemma for $C$]\label{cap2defC}
Let $(X,d)$ be a metric space, and $C\subset X$ be a compact set. Let
$\sigma\in\R^+$ and let $f:X\rightarrow \R$ be a conti\-nuous function
s.t.
\begin{equation}
|d_kf|(u)>\sigma \,\,\,\forall u\in C.
\end{equation}
Then there exist $\widetilde{C}\supset C$, $\tau\in\R^+$ and
$\eta:\widetilde{C}\times[0,\tau]\rightarrow X$ a continuous
functional such that:
\begin{itemize}
\item $\eta(u,0)=u$ for all $u\in\widetilde{C} $; \item
{$f(\eta(u,t))-f(u)\leq -\sigma t$ for all $u\in\widetilde{C}$,
$t\in[0,\tau]$}.
\end{itemize}
\end{lemma}
\begin{proof}
We know by hypothesis that $|d_kf|(u)\geq \sigma$, so for every
$u\in C$ there exist a $\delta_u>0$, a continuous map
$k_u:[0,\delta_u]\rightarrow\R^+$, $k_u(0)=0$, and a continuous function
\begin{displaymath}
\h_u:B(u,\delta_u)\times[0,\delta_u]\rightarrow X
\end{displaymath}
satisfying (\ref{cap2ksdef1}) and (\ref{cap2ksdef2}). By Milnor's Lemma
we know that the open cover
$\{B(u,\frac{\delta_u}{2}),\,u\in C\}$ admits a locally finite
refinement $\{V_{j,\lambda},\,\,j\in\N,\,\lambda\in\Lambda_j\}$
such that
\begin{displaymath}
\lambda\neq\mu\Rightarrow V_{j,\lambda}\cap V_{j,\mu}=\emptyset.
\end{displaymath}
By compactness of $C$ we can suppose that $\{V_{j,\lambda}\}$ be a
finite family. In particular there will be an $h_0$ and a finite
number of elements in $\Lambda_j$ s.t.
the family $\{V_{j,\lambda},\,\,j=1,\cdots,h_0,\,\lambda\in\Lambda_j\}$
covers the whole $C$.

Let $\vartheta_{j,\lambda}:X\rightarrow[0,1]$ be a family of
continuous functionals with
\begin{displaymath}
\text{spt }\vartheta_{j,\lambda}\subset V_{j,\lambda},
\end{displaymath}
\begin{displaymath}
\sum_{j=1}^{h_0}\sum_{\lambda\in\Lambda_j}\vartheta_{j,\lambda}(u)=1.
\end{displaymath}

For every $(j,\lambda)$ let $V_{j,\lambda}\subset
B(u_{j,\lambda},\delta_{u_{j,\lambda}})$. To simplify the
notations set $\delta_{j,\lambda}=\delta_{u_{j,\lambda}}$,
$k_{j,\lambda}=k _{u_{j,\lambda}}$ and
$\h_{j,\lambda}=\h_{u_{j,\lambda}}$. Let $\tau_0$ be a positive real
number such that $0<\tau_0<\min\delta_{j,\lambda}$, so every
$k_{j,\lambda}$ is well defined on $[0,\tau_0]$. Let
$$
k(t)=\bigvee_{j,\lambda}k_{j,\lambda}(t);
$$
let $\tau_1$ be a positive real number such that
\begin{equation}
\max_{t\in[0,\tau_1]}k(t)\leq\frac{1}{2}\,
\frac{\min\delta_{j,\lambda}}{h_0\sum_j\#\Lambda_j}.
\end{equation}
Set $\tau=\min\{\tau_0,\tau_1\}.$

Now, called
$$
\widetilde{C}=\bigcup_{j,\lambda}\,\overline{V}_{j,\lambda},
$$
we want to define a sequence of continuous map
\begin{displaymath}
\eta_h:\widetilde{C}\times[0,\tau]\rightarrow X
\end{displaymath}
such that
\begin{eqnarray}
d(\eta_h(v,t),v)&\leq& \sum_{j=1}^h\sum_{\lambda\in\Lambda_j}
k(\vartheta_{j,\lambda}(v)t),\label{cap2etah1}\\
f(\eta_h(v,t))-f(v)&\leq&
-\sigma\left(\sum_{j=1}^h\sum_{\lambda\in\Lambda_j}
\vartheta_{j,\lambda}(v)\right)t.\label{cap2etah2}
\end{eqnarray}

First of all we set
\begin{displaymath}
\eta_1(v,t)=\left\{
\begin{array}{ll}
\h_{1,\lambda}(v,\vartheta_{1,\lambda}(v)t),
&\text{if }v\in\overline{V}_{1,\lambda};\\
v, &\text{if }v\notin \bigcup_{\lambda\in\Lambda_1}V_{1,\lambda}.
\end{array}
\right.
\end{displaymath}
Obviously $\eta_1$ satisfies (\ref{cap2etah1}) and (\ref{cap2etah2}); now
we proceed by induction: assume that we have defined  $\eta_{h-1}$
satisfying (\ref{cap2etah1}) and (\ref{cap2etah2}). For every
$v\in\overline{V}_{h,\lambda}$ we have
\begin{displaymath}
d(\eta_{h-1}(v,t),v)\leq\sum_{j=1}^{h-1}\sum_{\lambda\in\Lambda_j}
k(\vartheta_{j,\lambda}t)\leq(h-1)\sum_j\#\Lambda_j
\max_{t\in[0,\tau]}k(t)\leq\frac{1}{2}\,\,\delta_{h,\lambda},
\end{displaymath}
hence $\eta_{h-1}(v,t)\in B(u_{h,\lambda},\delta_{h,\lambda})$, so
the map
\begin{displaymath}
\eta_h(v,t)=\left\{
\begin{array}{ll}
\h_{h,\lambda}(\eta_{h-1}(v,t),\vartheta_{h,\lambda}(v)t),
&\text{if }v\in\overline{V}_{h,\lambda};\\
\eta_{h-1}(v,t), &\text{if }v\notin
\bigcup_{\lambda\in\Lambda_h}V_{h,\lambda}.
\end{array}
\right.
\end{displaymath}
is well defined and satisfies (\ref{cap2etah1}) and (\ref{cap2etah2}).

Now we set
\begin{equation}
\eta(u,t)=\eta_{h_0}(u,t),
\end{equation}
so we have that $\eta:\widetilde{C}\times[0,\tau]\rightarrow X$ is
continuous. Furthermore
\begin{equation}
d(\eta(v,t),v)\leq \sum_{j=1}^{h}
\sum_{\lambda\in\Lambda_j}k(\vartheta_{j,\lambda}(v)t) \Rightarrow
\, \eta(0,v)=v,
\end{equation}
\begin{equation}
f(\eta(v,t))-f(v)\leq -\sigma \left(\sum_{j=1}^{h-1}
\sum_{\lambda\in\Lambda_j}\vartheta_{j,\lambda}(v)\right)t=-\sigma
t,
\end{equation}
that concludes the proof
\end{proof}

In this lemma we have used the compactness of $C$ to compound the local
retractions without using the property (\ref{cap2ws1}) of the weak slope.
To find a retraction on $X\smallsetminus C$ we must suppose that
$k_u(t)\leq t$ and proceed as in Degiovanni, Marzocchi and Corvellec work
\cite{CDM93}.

\begin{lemma}[deformation lemma for $X\smallsetminus C$]\label{cap2defX-C}
Let $(X,d)$ be a me\-tric space; let $\sigma\in\R^+$ and
let $f:X\rightarrow \R$ be a conti\-nuous function s.t.
\begin{equation}
|d_kf|(u)>\sigma \,\,\,\forall u\in X.
\end{equation}
Suppose also that there
exists a compact set $C\subset X$ such that
\begin{equation}\label{cap2kdiseg}
k_u(t)\leq t\hbox{ }\hbox{ }\forall u\in X\smallsetminus C,
\end{equation}
where $k_u$ is defined as in (\ref{cap2ksdef1}).

Then exist $\tau\in\R^+$ and $\eta:X\smallsetminus
C\times[0,\tau]\rightarrow X$ a continuous functional such that
\begin{itemize}
\item $\eta(u,0)=u$ for all $u\in X\smallsetminus C$; \item
{$f(\eta(u,t))-f(u)\leq -\sigma t$ for all $u\in X\smallsetminus
C$, $t\in[0,\tau]$}.
\end{itemize}
\end{lemma}
\begin{proof}
For all details see \cite[theorem 2.8]{CDM93}. We note only that the
proof is quite similar to lemma \ref{cap2defC}, but for proving that
$\eta_h$ is well defined we must use the inequality (\ref{cap2kdiseg})
to obtain a good estimate of $d(\eta_{h-1}(v,t),v)$.
\end{proof}

By lemma \ref{cap2defC} and lemma \ref{cap2defX-C} the proof of main theorem
follows as usual.

\begin{proof}[Proof of theorem \ref{cap2ksmain}]
Let $V\subset X$ be s.t. $C\subset V\subset \widetilde{C}$; we can
also choose $V$ such that $B(V,\rho)\subset\widetilde{C} $ for
some $\rho>0$. Set $\eta_C$ and $\eta_{X\smallsetminus C}$ the
retraction found respectively in lemma \ref{cap2defC} and
\ref{cap2defX-C}. For the sake of simplicity we suppose that they are
defined for all $t\in[0,1]$. Let $\theta:X\rightarrow[0,1]$ be a
continuous map s.t.
\begin{eqnarray}
\theta_1|_C&\equiv&0;\\
\theta_1|_{X\smallsetminus V}&\equiv&1.
\end{eqnarray}
and let $\theta_2=1-\theta_1$. Then we define a continuous map
\begin{displaymath}
\mu_1:X\times[0,1]:\rightarrow X,
\end{displaymath}
\begin{displaymath}
\mu_1(u,t)= \left\{
\begin{array}{ll}
\eta_{X\smallsetminus C}(u,\theta_1(u)t)&u\in X\smallsetminus C,\\
u&\text{otherwise};
\end{array}
\right.
\end{displaymath}
we know that
\begin{equation}
f(\mu_1(u,t))-f(u)\leq -\sigma t\theta_1(u),
\end{equation}
and that
\begin{equation}
d(\mu_1(u,t),u)\leq t\theta_1(u).
\end{equation}
Now let
\begin{displaymath}
\mu_2= \left\{
\begin{array}{ll}
\eta_{C}(\mu_1(u,t),\theta_2(u)t)& u\in V, \\
\mu_1(u,t)& \text{otherwise};
\end{array}
\right.
\end{displaymath}
we found that
\begin{displaymath}
d(\mu_1(u,t),u)\leq\theta_1(u)t\leq\rho\,\,\,\;\text{ if }t\leq\rho,
\end{displaymath}
so $\mu_2$ is well defined on $X\times[0,\rho]$.

Obviously we have that
\begin{equation}
\mu_2(u,0)=0\,\,\,\forall u\in X;
\end{equation}
furthermore, if $u \in X\smallsetminus V$ we have that
\begin{equation}
f(\mu_2(u,t))-f(u)=f(\mu_1(u,t))-f(u)\leq-\sigma t,
\end{equation}
and that, if $u\in V$, then
\begin{equation}
\begin{array}c
f(\mu_2(u,t))-f(u)=f(\eta_C(\mu_1(u,t),\theta_2(u)t))-f(u)=\\
=f(\eta_C(\mu_1(u,t),\theta_2(u)t))-f(\mu_1(u,t))+f(\mu_1(u,t))-f(u)\leq\\
-\sigma\theta_2(u)t-\sigma\theta_1(u)t=-\sigma t.
\end{array}
\end{equation}
So, we set $\tau=\rho$ and $\mu=\mu_2$ and we conclude the proof.
\end{proof}

We provide a final remark:
we observe that if there exist a compact set $C\subset X$,
then we are allowed to weaken the standard definition of weak slope.
to compound the local retraction explicitly.

In $X\smallsetminus C$ we must recover
condition (\ref{cap2ws1}) adding the hypothesis (\ref{cap2kustima}) of
theorem \ref{cap2ksmain}: in a non compact set this estimate
makes possible a continuous composition of local retractions.

\nocite{MS83}

\providecommand{\bysame}{\leavevmode\hbox to3em{\hrulefill}\thinspace}

\end{document}